\documentclass{amsproc}

  \usepackage[all]{xy}

  \usepackage{epsf,epsfig,amsfonts,graphicx,color}

\usepackage{amsmath,amssymb}
  \usepackage{url}
  
  \usepackage{textcase}

\newtheorem{theorem}{Theorem}[section]

\theoremstyle{definition}
\newtheorem{definition}[theorem]{Definition}

\newtheorem{oproblem}{Question} 
\newtheorem{proposition}{Proposition}[section]

\theoremstyle{remark}
\newtheorem{remark}[theorem]{Remark}

\numberwithin{equation}{section}

 \def\R{\mathbb R}

 \def \E {\mathbb E}

\def\eps{\epsilon}

\newcommand{\beq}{\begin{equation}}
\newcommand{\eeq}{\end{equation}}
\newcommand{\Leq}[1]{\label{#1}\end{equation}}

\begin{document}
\title[\NoCaseChange{Halfway to heaven}] {Halfway between Heaven and  Hell  \\[0.5em] {\small Questions arising while crossing the virial hypersurface}}


\author{Richard Montgomery}
\address{Math Dept., UC Santa Cruz,  Santa Cruz, CA, USA, 95064}
\curraddr{Math Dept., UC Santa Cruz,  Santa Cruz, CA, USA, 95064}
\email{rmont@ucsc.edu}
\thanks{Simon's Foundation Travel Grant}



\date{today}

\begin{abstract}
 We pose several   questions for the classical N-body problem  
 inspired by connections between the  virial equation and the Jacobi-Maupertuis formulation
 of mechanics.   We answer some.  
\end{abstract}

\maketitle


\section{Brake orbits and the Virial Hypersurface}


The classical  N-body problem is the original  example of   a natural mechanical system (see  \cite{AbMa}), by which we mean a dynamical system   determined
by choosing a potential function and a Riemannian metric on a manifold.  Such a system has conserved energy 
\beq E = \text{kinetic} + \text{potential} =  \frac{1}{2} \langle \dot q , \dot q \rangle - U(q). 
 \Leq{eq: energy} 
 and  evolves by    Newton's
 equations 
\beq \ddot q =  \nabla U (q).
\Leq{N1}
The moving point   $q = q(t)$ traces out a path in the manifold.  We write $\dot q$ for the 
  velocity of the path.  
   The first summand of
  the energy \eqref{eq: energy} is the kinetic energy  $K =  \frac{1}{2} \langle \dot q , \dot q \rangle  = \frac{1}{2} \| \dot q \|^2$.
  The second summand  is the potential energy   so that the function $U(q)$
  is the negative of the potential energy at $q$.   

The   Riemannian manifold for the N-body  problem 
is a Euclidean space $\E$.   As a vector space  
$\E = \R^3 \times \ldots \times \R^3  = \R^{3N}$.  We write points 
$q \in  \E$ as  $q = (q_1, q_2 \ldots , q_N)$ where  $q_a \in \R^3$ represents
 the (instantaneous)  location of the $a$th body.   Then a velocity $\dot q = 
 (\dot q_1, \dot q_2 \ldots , \dot  q_N) \in \E$ represents the velocities of all $N$  bodies.  
 The  inner product  on $\E$   is related to the standard expression for  kinetic energy $K$ by
 \[ 2 K(\dot q) =   \| \dot q \|^2 = \sum m_a | \dot q_a  |^2. \] 
 Write $r_{ab} = |q_a - q_b|$. Then
 \beq
  U(q) = G \sum_{a < b} \frac{m_a m_b}{r_{ab}} .
  \Leq{potential}
  is the  negative of the standard  potential function used for the defining the N-body problem.  We have that    $U : \E \to \R \cup \infty$.  
   Points $q$ at which  $U(q) =   + \infty$ are  collision points as they   represent configurations where  two or more  of the N bodies 
have collided (some $r_{ab} = 0$).  
 $U$ is continuous  on $\E$ and  analytic away  from collisions. It is positive everywhere.   The coefficient $G >0$ is
 called the universal gravitational constant and is needed for $U$ to have units of energy. 
 The gradient   $\nabla U$ of equation \eqref{N1} is the gradient of $U$ with respect to the inner product on
 $\E$ and is well-defined and analytic everywhere except at collisions.

A `brake instant' for a solution $q(t)$ to equation \eqref{N1}  is a time $t_*$ such that $\dot q (t_*) =0$:    all    N bodies are instantaneously at rest. 
A    `brake orbit' is a    solution   with  a brake instant.   Note that the kinetic energy $K = 0$ at a brake instant.
It follows that the   energy $E$ of 
a brake orbit is  a negative constant  $E = -h$  where   $h = U(q(t_*)) > 0$ is  the value of  $U$  at the brake instant.

Brake orbits are easy to come by using a numerical integrator.  Integrate  \eqref{N1} with initial conditions for which
$\dot q = 0$.    Brake orbits typically 
come very close to the collision locus $U = + \infty$, this being the locus
where two or more of the bodies have collided.   
  You can find a popular account of brake orbits in \cite{Dropping} and a detailed studies  in \cite{Li}
  and \cite{Hristov}, where tens of thousands of periodic three-body brake orbits were found.
  You can find a proof of the existence of countably infinite families of periodic isosceles brake
  orbits for three bodies in \cite{Chen}.

Any solution $q(t)$,  brake or not, with energy  $E  = -h$ 
must satisfy the constraint  $U(q(t)) \ge h$ as follows immediately
 follows  from $E = K - U$ and $K \ge 0$.   
 We call  the domain $\{q : U(q) \ge h \} \subset \E$  
 the {\it Hill region}.  Energy $E = -h$ solutions $q(t)$ must travel inside the  Hill region.

\begin{theorem} [Lost Theorem]  Consider  the Newtonian N-body problem at fixed negative energy $E = -h$.
Let $q_0 \in \E$ be a point  in the corresponding Hill region  $\{U  \ge h \}$.    There exists an  energy $E$ brake orbit passing through $q_0$. 
\end{theorem} 
\noindent  For a proof of the last theorem see the end of subsection \ref{sec: recollections}
below.  See figure \ref{fig: HillRegion2} for a depiction of the Hill region for the planar three-body problem.

Why call this   a lost theorem?  It is  found in my   paper    \cite{Filling} and
earlier   in the joint paper   \cite{{Brake-to-Syz}}. In the  intervening  12 years
between these two papers my co-authors and I lost and forgot this theorem.  
The refound theorem  above was a by-product of month or so  of   Zoom discussions with a group of 
colleagues during the pandemic.  

The  brake solutions  guaranteed by  the theorem     have energy $E = -h$ and  $U  = h$ 
at the brake instant 
since   $K  =  0$ at this instant.     We call the hypersurface $\{ U = h \}$ the {\it Hill boundary}.
All energy  $E = -h$ brake orbits hit their  Hill boundary orthogonally, and then retrace their path:
$q(t_* + s) = q(t_* - s)$ where  $t_*$ is the brake instant.

\vskip .3cm

Deep inside  the Hill region, at the   other extreme
from  the Hill boundary,    lies the   collision locus $\{U = + \infty \}$. 
 Halfway between
the  level sets $\{ U = h\}$ and $\{U = + \infty \}$  
lies the level set defining the  
\beq
\text{ virial surface:} = \{q : U(q)  = 2h \}
\Leq{eq: virial}
Numerical experiments and   common sense 
suggest  that all  brake orbits   get  quite  close to the collision locus.   If this is true,
then all brake orbits must cross the virial hypersurface.    There are many periodic
brake orbits (\cite{Dropping, Li}) and  all of them      cross the virial hypersurface as
follows from 
\begin{theorem}[Virial theorem, periodic version]  If $q(t)$ is a periodic solution to Newton's N-body equation  
then its total energy $E = -h$ must be negative. 
Moreover the average $\langle U \rangle$ of $U(q(t))$ over one period is $2h$   so that 
 $q(t)$ must oscillate  about the virial surface in the sense that it must
cross the   level set $U = 2h$ at least twice per period.
\label{periodic virial}
\end{theorem}
\noindent See figure \ref{fig: schematic} for  a depiction of the theorem.

 \begin{figure}
  \includegraphics[width=8cm]{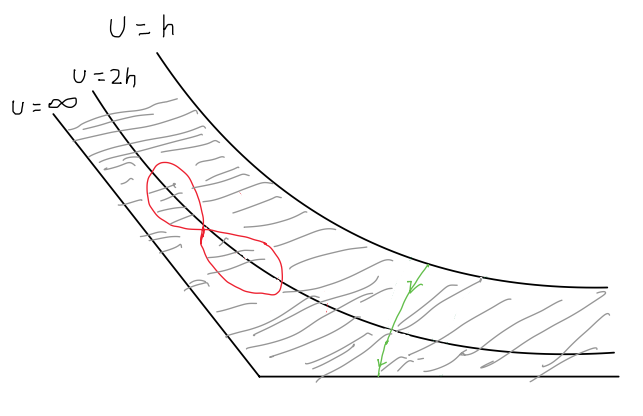}
 \caption{A schematic of the Hill region, shaded with a collision brake orbit
 (green) and non-collision periodic orbit (red) indicated.  } 
 \label{fig: schematic}
 \end{figure}

Here is the  standard proof of the virial theorem.   See Landau-Lifshitz \cite{Landau},   equation (10.5), or the wiki page on the virial equation.

\vskip .4cm

\begin{proof} 

Differentiate $I = \langle q, q \rangle$
twice along the orbit. Once:
$$\dot I = 2 \langle q, \dot q \rangle, $$
twice: 
\begin{eqnarray*}
\ddot I & = 2 \langle \dot q, \dot q \rangle + 2 \langle q, \ddot q \rangle \\
 & = 4 K(\dot q)  + 2 \langle q, \nabla U (q) \rangle \\
  & = 4 K(\dot q)  - 2 U(q)  \\
\end{eqnarray*}
The final  identity  for $\ddot I$ is known as the `Lagrange-Jacobi identity'.  In obtaining the final  line of the identity we used the  fact that $U$ is homogeneous of degree $-1$ 
and Euler's identity for homogeneous functions to replace  $\langle q, \nabla U (q) \rangle$ with $- U(q)$. 

Now $\dot I$ is periodic so the average of $\ddot I$ is zero
and so we can average $K$ and $U$ over the orbit    to get $4\langle K \rangle  -2 \langle U \rangle = 0$
where we  have written  $\langle f \rangle$ for the time-average of  a function $f$ on phase space $\E \times \E$
over the curve $(q(t), \dot q(t))$ in phase space. 
Use that the conserved  energy is $E = K - U = \langle K \rangle - \langle U \rangle$
to conclude that that $4E + 2\langle U \rangle = 0$.
or    $\langle U \rangle = 2h$ where $h = -E$.  In order for $U(q(t))$ to have average $2h$
the solution $q(t)$ must  either lie identically on the virial surface or cross it twice per period. 
And since $U > 0$ we have $\langle U \rangle > 0$ so $E < 0$. 

\end{proof}

 This theorem  suggests  
\begin{oproblem}
Does {\bf every}  solution having negative energy $E = -h$ and defined for all time  cross the virial hypersurface $U = 2h$?
\end{oproblem} 

{\sc Answer:}  No.  Indeed, in the interim between drafts of the currect paper Moeckel \cite{MoeckelLarge} proved:

\begin{theorem}  [Moeckel \cite{MoeckelLarge}] Given any positive constants $C$ and $h$  there are collision-free solutions $q(t)$ of the 
planar three-body problem  having  energy $E = -h$ and for which $U(q(t)) \ge C$  for all $t$.
\label{Moeckel thm}
\end{theorem}
 
We will  return to  Moeckel's theorem  in  section \ref{sec: escape} below.
Let us now continue to follow the virial thread.  Any one of the three  equivalent  equations $2\langle K \rangle  -  \langle U \rangle = 0$ , 
$ \langle U \rangle = -2E $ and  $ \langle K \rangle =  E $ which we  proved when we proved theorem \ref{periodic virial}  are   known as the virial equation.
The virial equation, with averages properly formulated,   is  also   valid for aperiodic solutions under a rather large
range of circumstances.   However, the virial equations  are clearly not valid for Moeckel's solutions since  
whatever the average $\langle U \rangle$ might mean,  we must  have that  $\langle U \rangle \ge C$  if $U(q(t)) >C$ for all $t$,
so we cannot have $\langle U \rangle = 2h = -2E$.

 To define  the average 
$\langle f \rangle$ of a function $f$ such as $K$ or $U$ on phase space $\E \times \E$ over an aperiodic   solution $q(t)$
we can  average $f$ over increasingly long orbit segments of $q(t)$ and
take the limit:
$$\langle f \rangle = \lim_{T \to \infty} \frac{1}{2T}\int_{-T} ^T f(q(t), \dot q(t)) dt$$
a limit which  may or may not exist.  For our purposes the state-of-the-art virial theorem  was proven  by Pollard \cite{Pollard}. 

\begin{theorem} [Pollard's Virial theorem]   
Let $q(t)$ be a    solution to Newton's equations  defined for all time.
If $I(t) = o(t^2)$ as $t \to \pm \infty$
then both averages $\langle K \rangle$ and
$\langle U \rangle$ over $q$ exist as defined immediately above and they  satisfy  the virial equation.
Conversely, if these  limiting averages exist and satisfy  the virial equation   along $q$ 
then  $I(t) = o(t^2)$ as $t \to \pm \infty.$
In particular, the virial equation  holds for all   bounded collision-free solutions. \label{Pollards virial}
 \end{theorem}

 Pollard's $I(t) = o(t^2)$ required bound
 holds  for all  bounded solution ( $I(t) = O(1)$) in addition to the periodic solutions. 
 It also holds for   parabolic solutions
($I(t)  = O(t^{4/3}))$.  For   parabolic solutions the energy $E$ is zero and
 so  we get   $\langle K \rangle = \langle U \rangle = E = 0$.   For a bounded solution
  we have $U > C$ for a positive constant and thus  $\langle U \rangle > 0$ and the solution's energy is  $E = -h < 0$.  This
  discussion  suggests 
   a modification
  of our first problem.
  
  \begin{oproblem}
Does {\bf every}  bounded   solution   cross its virial hypersurface $U = 2h$?  
\end{oproblem}

{\sc A possible `no'.}  Here is a cartoon  scenario suggesting the existence of  a solution  which answers  `no'   to question 2.    The bell curve of statisitics
provides an example of a function which is  positive everywhere but whose  average over the  real line is $0$.  In a similar way we could have a bounded solution
$q(t)$ which is heteroclinic  between  two unstable rotating relative equilibria at the given energy $E = -h$.  Both relative equilibria have constant $U$
and so have    $U = 2h$ by the periodic virial theorem.  But it could  happen that $U(q(t)) > 2h$ all along
our heteroclinic solution despite the fact that $lim_{t \to \pm \infty} U(q(t)) = 2h$ and $\langle U \rangle = 2h$.

    \begin{remark} Pollard's equations (6) and (8) are hard to read in the JSTOR copy
  since the dots do not show up in the copy.  In particular the left hand side of (6)
  is $\ddot I$ not $I$.   (This equation is a version of the Lagrange-Jacobi equation.)       \end{remark}

   \begin{remark} 
   Pollard's  bound, $I(t) = o(t^2)$ fails for Moeckel's solutions which 
  instead   satsify $I(t) = C t^2 + o(t)$ for $t$ large and for a positive $C$.  
As  discussed already    the virial equation does not  hold for such solutions and  
instead we get  that $2\langle K \rangle  -  \langle U \rangle > 0$.
\end{remark}

  \begin{remark}
  {\sc A modified hyperbolic virial.}  Moeckel's solutions are of hyperbolic-elliptic type.
  Such solutions have one body escaping to infinity while the remaining pair remains bound  in 
  a near-Keplerian periodic orbit about its center of mass.  The (Jacobi) vector connecting the escaping body 
  to this center of mass of the bound  has a limiting  non-zero velocity $v_{\infty}$.  
   A bit of analysis and some algebra shows that
  $$2 \langle K \rangle  - \langle U  \rangle = 2K_{hyper} >0 $$
   where $K_{hyper} = \frac{1}{2} \mu \| v _{\infty} \|^2$.
 The constant  $\mu >0$ is the usual  
  reduced mass corresponding to the kinetic energy splitting induces by
  our Jacobi vector choice.  The averages of phase functions $f$ are redefined as $\lim_{T \to \infty} \frac{1}{T} \int _0 ^T f(q(t), \dot q(t))dt$
  so as to  only concern  the positive time  escaping half- trajectory.    If 
  both halves  of the orbit are  hyperbolic-elliptic escape orbits and if they have  different asymptotic velocities and hence different limiting hyperbolic escape kinetic energies 
  $K_{hyper} ^+ ,  K_{hyper}^-$   
  we get  $2 \langle K \rangle  - \langle U  \rangle = K_{hyper} ^+ +  K_{hyper}^- > 0$
 upon reverting to the original $\lim_{T \to \infty} \frac{1}{2T} \int _{-T} ^T f(q(t), \dot q (t)) dt$ definition of the average.
 \label{hyperbolic-elliptic}
   \end{remark}
  
    \begin{remark} The virial equation and its many interpretations  has
been   important in the development of statistical mechanics, galactic dynamics
and   the discovery of dark matter.  See the wiki page on ``virial equation'' as of 2026. 
  \end{remark}

\vskip 1cm
\subsection{More virial musings} 
{\it  Are there  
solutions which   live their whole  lives on the
virial hypersurface? }

{\bf Yes.} The relative equilibria are periodic solutions for which 
$U = 2h$ throughout.    These are  solutions which evolve by rigid rotation,
each body tracing a circle about their common center of mass.  The
fiducial example is the  Lagrange solution for $N=3$  whose time evolution is that
of an equilateral triangle evolving  by  
rigid rotation.   For any $N$ and any mass distributions
there exist   relative equilibria, the number of them growing with $N$.  They correspond to
  planar central configurations. See \cite{MoeckelCC}.

{\it Are there any other solutions besides the relative equilibria which stay on
the virial hypersurface?   } 

 A conjecture of Saari, proved for N = 3 by Moeckel \cite{SaariConj1} 
asserts that `{\bf no}' ,  the  relative equilibria  are the only solutions which
satisfy $U(q(t)) = 2h$ for their whole existence.     

\vskip .6cm

\subsection{Kepler-inspired virial annuli} 
 Each relative equilibrium is one solution   of a one-parameter  {\it Keplerian family} of
 homographic solutions. In these solutions each body travels a Keplerian ellipse, with 
 the entire configuration evolving by  a periodic  time-dependent  scaling and rotation.
 We can parameterize the elements of this 
 family  by their angular momentum, with the value at which they are relative
 equilibria being the maximum  value of angular momentum $|J|$ in the family for the given fixed (negative)  energy. 
   If we keep track of $U(q(t))$
 along such a  Keplerian family   we find 
\beq
 \frac{2h}{1+k}   \le U(q)  \le \frac{2h}{1 - k}.
 \Leq{U constraint} 
 with the endpoints achieved and where $k = k(J)$ with  $k \to 1$ as $J \to 0$.
 The elements of the family having $J = 0$ and so $k =1$
 are brake solutions ending in total collision. 
 \begin{definition}  Call the region defined by \eqref{U constraint} the virial annulus
 ${\mathcal A}_k$. 
 \end{definition} 
   
  \begin{definition}  Call the {\it thickness} of a solution $q(t)$  to \eqref{N1}
 the smallest positive  $k$ such that \eqref{U constraint} holds along the entire solution
 \end{definition}

  These considerations suggest 
 \begin{oproblem}   Introduce the real parameter $k \in [0,1]$.
 Consider the space  ${\mathcal B}_k$  of all energy $E = -h$ bounded solutions
 which satisfy the constraint \eqref{U constraint}. 
 As $k$ increases   how does  the ``number'' and ``type'' of different 
 orbits in ${\mathcal B}_k$ grow?   By what mechanisms do  new solutions get added   to 
 ${\mathcal B}_k$ as $k$ increases?  
 \end{oproblem}
 
 Suppose that  the bodies move in the plane instead of space
 so that we replace our Euclidean space $\E = (\R^3) ^N$ with  $\E = (\R^2) ^N$.
 Then the   interior of the Hill region minus collisions has a large and rich fundamental group $\pi_1$, namely the pure braid group on N strands.
 (See \cite{MeBraids}.)    As we explain in the next paragraph, this entire group  is supported on the virial surface
$VS =  \{ U = 2h \}$ or any of the virial annuli ${\mathcal A}_k$, in the sense that any representative  homotopy class
for any element of $\pi_1$  can be
isotoped so as to lie in $VS$ or in  ${\mathcal A}_k$.
Thus either the virial surface  or any  virial annulus  admits  closed curves
realizing any given braid type as encoded by $\pi_1$.  
  The question above is asking  how many of these braid types, as a function of $k$ are realized by periodic solutions
in ${\mathcal A}_k$.    

We now show that    $\{ h \le U < \infty \}$
 is diffeomorphic to $VS \times I$ 
 where
 $VS =  \{ U = 2h \}$ is the virial surface of equation \eqref{eq: virial} and where $I$ is an interval.  To see this 
 use the homogeneity $U(\lambda q) = \frac{1}{\lambda} U(q)$.
 The scaling map   $(q, \lambda)  \mapsto \lambda q$ with $q \in VS$ 
maps $VS \times (0,  2] \to \E$ onto $\{ h \le U < \infty \}$ and is a diffeomorphism.
(The  map takes $VS \times \{\lambda \}$ to the level set $\{ U = \frac{2h}{\lambda} \}$.)
It follows that the interior of our Hill region   deformation retracts onto   $VS$ or onto any one of the  virial annuli ${\mathcal A}_k$
and hence both of these support all of $\pi_1$ as discussed above.

\begin{remark}[A Keplerian ruler] Within the Hill region $U$ can take on any value between $h$ and $\infty$.
A good ``Keplerian'' ruler for range $[2h, \infty]$ of $U$ in the Hill region is found by
relating  the scaling $\lambda \in (0,2)$  to the thickness $k$ by writing $\frac{1}{\lambda} = \frac{1}{1+k}$
or $k = 1+ \lambda$, thus rescaling the  $\lambda$-interval $(0,2]$ of the paragraph immediately above
to  the $k$-interval $(-1,1]$ appropriate to the Kepler constraint \eqref{eq: virial}. 
Measured with the  $k$-ruler  the value $k = -1$ corresponds to the collision locus and  $k =1$ corresponds to the Hill boundary
while $k = 0$ is the virial surface. 
\end{remark}

 The relative equilibria, lying on the virial surface,  represent the center of $\pi_1$.
 My main  interest here is in  how the subset $C_k$ of conjugacy classes of $\pi_1$ 
 realized by periodic solutions of thickness $k$  grows with $k$.    The figure eight
 with zero angular momentum has a quite small thickness.  We have the Hurewicz Abelianization
  map from $\pi_1 \to H_1 (VS)$ (where $H_1$ denotes homology).  The figure eight lies in the kernel of this map, being null-homologous.
 Moeckel and I showed that eventually all braids do get realized for $N=3$. See \cite{allRealized}.
The solutions we established are near collision, so have $k$ arbitrarily close  to $1$.

 I  have  no clue how to attack  question 3. 
My knee-jerk reaction is  to   build a  variational set-up  centered around the Jacobi-Maupertuis metric
to attack the question.  The point of the question  is to  view the thickness  $k$ of a solution
as an alternative   bifurcation parameterto the   traditional angular momentum
 $J$ (cf. Moeckel \cite{MoeckelQual}).  
As long as thickness  $k$ is less than $1$  the constraint  \eqref{U constraint} excludes both
 brake and collision solutions from being in ${\mathcal B}_k$.    It is not until $k = 1$ that we get all solutions. What  if $k = 1 - \eps$
 for some fixed  small positive $\eps$?   Might  it be that all possible   topological types of periodic  curves
 are realized by curves of this  thickness?
 
 One could add more parameters in addition to thickness :  the angular momentum, the maximum of $I$ along an orbit,
 or the degree $\alpha$ of homogeneity of $U$.   
 These additions could make the consequent  `realizing of $\pi_1$'' questions easier or harder.  
 
 \begin{remark} If we replace the standard Newtonian potential with one which is homogeneous of degree  $\alpha = 2$ 
 these  questions become  particularly easy and special and are fully solved. The modified
 Lagrange-Jacobi equation becomes   $\ddot I = 4H$, consequently the virial hypersurface degenerates.   Bounded solutions
must have  
energy $H = 0$ and  $I(t) = const$.    Almost all braid types  are realized by  periodic solutions. 
 See \cite{MeBraids}, \cite{hyperbolicPants}.   
\end{remark}

\section{ The Jacobi-Maupertuis Metric }

\subsection{Recollections and a proof} 
\label{sec: recollections}
We recall the Jacobi-Maupertuis [JM]  approach to mechanics.   Assume the system
evolves by equation \eqref{N1}.  Focus our attention on solutions
having a fixed energy $E = -h$.  
Form the metric
\beq ds^2_E =  2(E   - V(q)) ds^2 _{\E} = 2(-h   + U(q)) ds^2 _{\E}
\Leq{JM}
  on the Hill region $\{h \ge U\}$ (which equals 
$\{ -\infty \le V \le E \} $ where $V = - U$ is the potential, and  
where $ds^2 _{\E}$ is the mass metric viewed as a Euclidean (flat) Riemannian metric.

 \begin{figure}
  \includegraphics[width=8cm]{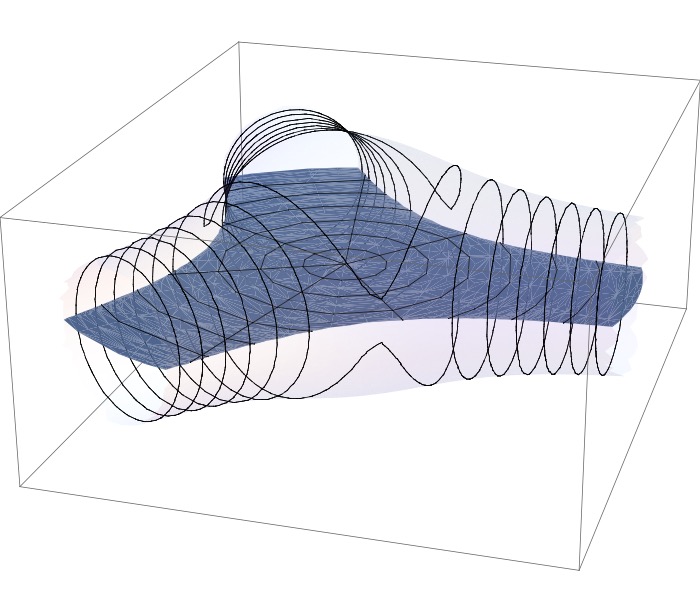}
 \caption{(Courtesy of Rick Moeckel.)  The  Hill region $\{ U \ge 1\}$ projected onto three-body shape space
 resembles  a   plumbing fixture made of  three pipes joined  at the origin.   
 The origin represents triple collision. The three  rays issuing from the origin
 about which  each pipe is centered  represent the binary collision locus.  For details regarding shape space and this picture see \cite{shape}.} 
 \label{fig: HillRegion2}
 \end{figure}

 \begin{theorem} [Jacobi-Maupertuis principle]  Away from   collisions  and brake  points, the energy $E$  solutions to Newton's equations
are reparameterizations of  geodesics for the Jacobi-Maupertuis metric (\ref{JM}) on the    Hill region.     Conversely, away from the collisions  and   the Hill boundary,   geodesic
for this Jacobi-Maupertuis metric are reparameterizations of 
 energy $E$ solutions of Newton's equations.
 \end{theorem} 
 
 \noindent See 
 Landau-Lifshitz \cite{Landau} p. 141 or    Knauf \cite{Knauf}  p. 178  or Abraham-Marsden \cite{AbMa} p.228  for proofs of the Jacobi-Maupertuis principle.
 See \cite{Whos Afraid} for details of the dynamics near the Hill boundary.
 
 Complete the metric space  on the interior ($ h < U < \infty $) of the
 Hill region which is induced by the (incomplete)  Riemannian  JM metric on this interior.  
 What we  get  is a somewhat strange metric on the entire   Hill region with the  brake ($h = U$)  and collision 
 ($U = \infty$) points added in.
 Travel along the Hill boundary is free: paths on the Hill boundary have zero length relative to the JM metric.
 As a consequence the entire Hill boundary  must be collapsed to a point when we complete the metric.
 We call this single point {\it the brake point}.  
 On the other hand, collisions remain a part of the completed space since there are paths
 of finite JM length touching any collision point. 
 \begin{theorem}  The Jacobi-Maupertuis [JM] metric for the N-body problem at negative energy has  finite diameter.
\label{thm 1}
\end{theorem}

We prove this theorem in \cite{finiteDiam}.
From the theorem   we can prove:

\begin{proof}

 [of the Lost Theorem] 
Let $q_*$ be a point in the interior of the Hill region,
so   that  $\infty > U(q_*) > h$.  Use  the direct method of the calculus of variations
to  find a path $c$   in $M_h$ joining $q_*$ to the brake point $B \in M_h$
 which  minimizes  the  JM length among all such paths.   A variant of Marchall's lemma (\cite{Brake-to-Syz})
 guarantees  that this length-minimizing curve
has no  collisions.       
It follows that this curve  is a Riemannian geodesic and after 
  reparameterization becomes an  energy $E$ solution
to   \eqref{N1} 
\end{proof}

\subsection{Mountain Pass Dreams}

Take any free homotopy class for the planar N-body problem.
We can represent it by a curve of zero length  lying on the Hill boundary.
Now apply our scaling to the curve:  $q_{\lambda} (t) = \lambda q(t)$.
As $\lambda \to 0$ the curve shrinks to total collision at $q =0$.
Its JM length  also shrinks to zero.  This is most easily seen
by first  isotoping the curve to lie on a small    sphere $r = const$ within the interior of the
Hill region. Here   $r^2 = I = \|q \|^2$.  We can   write the potential as $U(r \omega) = \frac{1}{r} \hat U (\omega)$
where $\hat U$ is homogeneous of degree zero and so represented as a function on the
sphere $r = 1$.  Write the kinetic energy metric in spherical coordinates using $q = r \omega$, with  
$dr ^2 + r^2 d \omega^2$ and  $d \omega^2$ is the standard metric, or squared arclength,  on
the unit sphere $r = 1$.  Then, since $dr = 0$ for our isotoped curve $\hat c$,  the energy $E = -h$ Jacobi-Maupertuis
arclength $ds_{E}$ of this curve  satisfies   $ds_{E} = \sqrt{2(-h + U(r \omega)} (r d \omega ) \le   r^{1/2} \sqrt{2 \hat U (\omega)} d \omega $.
Now  consider the case $h = 0$ and the  length of our isotoped curve, normalized to lie on the  unit sphere $r =1$.
This length is  some constant $C = \int_{\hat c} \sqrt{2 \hat U (\omega)} d \omega $.
Consequently our shrinking curve, placed on the sphere of radius $r $ (small),  has length less than or equal to  $r^{1/2} C$
which goes to zero with $r$.   Our precisely scaled curve $\lambda q(t)$,   does not lie on a fixed sphere,
but as it shrinks it lies within a shrinking spherical annulus $c_1 \eps < r < c_2 \eps$, since the curve is compact.   The error terms resulting
from the $dr^2$ term of the metric are easily verified to not change the asymptotics of the collapsing length of our curves.
Consequently the  JM length  of $q_{\lambda} (t)$ goes to zero as $\lambda^{1/2}$ with  $\lambda \to 0$.   

This scenario suggests   a  mountain pass
set-up for finding periodic solutions in a given free homotopy class.
Consider one-parameter families $q_{\lambda} (t)$ of curves, where $\lambda \in [0, 1]$
is the parameter.  Suppose at one endpoint   $\lambda = 1$ the curve lies completely
on the    Hill boundary,  while at $\lambda = 0$ we have $q_0 (t) \equiv 0$,
and that   in between, for $0 < \lambda < 1$ the curve $q_{\lambda}$ lies entirely
within the interior of the Hill region and realizes the chosen free homotopy class.   Take
the curve of maximum length in such a family. Presumably this curve crosses the virial hypersurface.
Maybe.  Anyhow, every such one-parameter family must cross  ``the  mountain pass''
joining our two  zero  JM-length  valleys consisting of curves supported on the  Hill boundary and curves
shrinking to  total collision. 
Now take the minimum (infimum)  of all these  maxima over all such one-parameter families.  Show that
the resulting min-max curve 
\begin{itemize}
\item{(i)} exists 
\item{(ii)} is collision-free
\item{(iii)} avoids the Hill boundary
\item{(iv)} extremizes JM length
 \end{itemize}
 Such a min-max curve would be a reparameterization
 of a periodic solution to Newton's equation in the desired free homotopy class.

\begin{oproblem}
Can you get a  mountain pass scheme  to work in order to establish  a   min-max  solution
realizing any given  braid type for the planar N-body problem?  
\end{oproblem}

\begin{remark}  Moeckel and I \cite{allRealized} established
solutions in every braid type for the planar 3-body problem.   We did so using
dynamical methods, not variational ones.  Our  solutions,
all  have thickness close to $1$ :  they come very close to collision
and to brake, being that they shadow very eccentic central configuration solutions
during parts of their motion.  I hope  a mountain pass scheme
might give us small thickness solutions realizing rather arbitrary braids on 3 strands.
\end{remark} 
 
 \section{Escape solutions}
 
Each virial annulus \eqref{U constraint} splits the Hill region in two,
namely  into the   components,
  $\{q:  h \le U(q) \le 2h/(1+k) \}$ and $\{q :  2h/(1-k) < U(q) \le + \infty \}$.
 Any solution which lies entirely within   one component   must be unbounded.
 Moreover  the virial equation must fail for these solutions. 
 Do both components  support unbounded solutions which exist for all time? For any $k$?
 No.
 
 \begin{proposition} There is an $\eps  > 0$  such 
 any energy $E = -h$ solution starting in the region $\{ h \le U \le h + \eps \}$
 leaves that region in finite time.
 \label{Hill collar}
 \end{proposition}
 
 \begin{proof}   [Sketch]  If we put together the results of \cite{Whos Afraid}
 and \cite{Filling}  we can establish the following more detailed
 picture of dynamics near the Hill boundary.   There exist   constants  $\eps_0 > 0$, 
  $K > 1$ and   $C > 0$ with the following significance.
  For all   $\eps < \eps_0$ the following holds.  Any energy $E = -h$ solution
 which enters into the region  $\{ U \le h + \eps \}$  will exit  the region $\{ U \le h + K \eps \}$
 by crossing  the hypersurface $U = h + K \eps$ transversally.  This crossing is transverse
 in both forward and backward time. The   exiting (and hence entering)
 occurs within a time of less than   $C \sqrt{\eps}$.  
 The proposition we stated follows by taking $\eps = \eps_0/2$. 
 
 The idea of the proof of this
 more detailed assertion of the previous paragraph is to use 
 the fact that $h$ is a regular value of $U$, and uniform estimates on $\nabla U$ and
 $d^2 U$ to turn the dynamics near the hill boundary to approximately the
 dynamics of a thrown ball.  We choose ``Seifert'' coordinates $(x_1, \ldots ,x_{n-1}, y)$
  near points of the Hill  boundary
 as described in \cite{Whos Afraid} so that the Hill boundary is given by
 $y = 0$, so that $U - h = y + $(error) and the energy $E$
 dynamics looks roughly like that for the potential $y$ in coordinates
 $(x_1, \ldots, x_{n-1}, y)$ where kinetic energy is $\sum dx_i ^2 + dy^2 + $(error).  
 
 \end{proof} 
     
 As $k \to 1 $ we have $2 h /(1+k) \to h$, so we by taking   $k$    close to $1$ 
 we guarantee that 
 $2h /(1+k ) =  h + \eps$ where $\eps$ is as in  proposition \ref{Hill collar}.  
  The  proposition now  guarantees that energy $h$ solutions starting in the component 
  $h \le U \le 2h/(1+k)$ 
 cannot stay in this  component.   
 
 \vskip .5cm 
What about   the other component,  $\{ 2h/(1-k) < U \}$, the one which contains
 the collision locus?   Can   solutions
   remain there for all  time?  (Note   
 as $k \to 1$ we have $C = 2h/(1-k) \to \infty$.)   
 Moeckel's theorem \ref{Moeckel thm}  above 
 asserts that the answer is `yes', at least for the three-body problem.
 Doubtless it is true for more bodies as well. 
 
 \section{Escape scenarios} 
 \label{sec: escape}
 Both  the 
 scenarios  we will describe yield  two-sided hyperbolic-elliptic escape solutions
 as described above in the remark \ref{hyperbolic-elliptic}.
 \subsection{Birkhoff-Moeckel's escape} 
  
 Moeckel's proof of his theorem \ref{Moeckel thm}   refines  
 an argument   found in Birkhoff \cite{Birkhoff} on  
  p. 278-282 of chapter IX.8.  Observe that 
  $I = \|q \|^2$  measures both 
  closeness to triple collision and escape.  If  $I(q)  \to 0$  then  $q \to 0$ which represents  triple collision,
  the origin  when we restrict to the center-of-mass zero subspace of $\E$.
  And when   $I(q) \to \infty$ while $U(q) \ge h = -E$   we have that the distance  of  one of three bodies 
  relative to the center of mass of the other two is tending to infinity.  The other two bodies  must stay within a bounded 
  distance of each other because
  of the negative energy constraint.  
  
   Introduce the conserved total angular momentum $J$.
     A result going back to Sundman asserts that for the three-body problem
  triple collision   is impossible along solution arcs  with  non-zero  angular momentum.
  However   families of initial conditions having
  fixed $J$ and $E$ may have  $I$ as small as we wish. 
  Moeckel's argument involves such families of initial conditions.

  Call a turn-around point along an orbit a point (or a time) at which  which $\dot I = 0$.
  Write $I_0$ for the value of $I$ at turn-around.  Moeckel  chooses  initial conditions
  for solutions at time $t = 0$ 
  having fixed $J \ne 0$, fixed $E = - 2h < 0$ and a turn-around point
  with very  small $I_0$ at turn-around.    Using McGehee blow-up of triple collision
  Moeckel shows that if   
  $I_0 <J^2 / 2h$ then  $\ddot I (0) > 0$.  Such a  solution then  satisfies   $I(t) \ge I_0 + a t^2$ on some   interval $[-\delta, \delta]$.   
  He then follows the argument in  Birkhoff to arrange that certain quantities are large enough at $t = \pm \delta$
  to insure escape.  This means that   $I(t)$ increases   monotonically for $t > \delta$ and hence on all of $[0, \infty)$
  with  $I(-t)$ similarly 
  increasing  strictly monotonically for $t > 0$.  So, overall $I(t)$ has a unique minimum at turn-around
  and tends to $\infty$ as $|t| \to \infty$.  
  
  Moeckel's   solution family is an   open non-empty family.
   To insure the solutions actually    exist for all time he can  either allow   regularizations of binary collisions a la Levi-Civita
  or argue that binary collisions
  are of measure zero to simply avoid all collisions in a full measure  sub-family of his   family. 
  
   \subsection{Sitnikov's escape} Recall  that the   original Sitnikov
   problem \cite{Sitnikov}, with its wonderful example of chaos and escape is a   restricted spatial isosceles
   three-body problem in which the two equally massive primaries
   move in slightly eccentric Keplerian orbits.  Make Sitnikov's problem  unrestricted, which is to say   consider
   the  spatial isosceles three-body problem with non-zero angular momentum and masses $m_1 = m_2 = m$. 
   At each  instant
   the three  bodies form an isosceles triangle with vertex $m_3$     lying on the z-axis,
   and   bodies 1 and 2  instantaneously spinning about this axis, spaced
   at the same height and at the same distance from the axis.    This symmetric subsystem
   of the three-body problem forms a   two-degree of freedom Hamiltonian system.
  (See \cite{Chesley}.)
     
  Spatial isosceles solutions can easily be constructed which are  of hyperbolic-elliptic type in either  time direction.
 The   third body would arrive  at infinity at a finite speed.
 Suppose that the  12 bound pair forms asymptotes to a circular  binary Keplerian orbit   moving in a   circle of radius $a$.  
  We can certainly construct such orbits in a  one-sided ($t \to + \infty$) fashion while   insuring  the total  energy $E = -h$ is negative by making the tight binary of sufficiently
  negative energy, which means taking   $a$ sufficiently small relative to the kinetic energy of the third body. 
  \begin{oproblem}  Let $a$ be any positive  radius  with $\frac{2m}{a}  > h$.
  Are there   energy $E = -h$ collision-free   solutions to the spatial isosceles three-body problem with all three masses
  nonzero  and  for which the 12  binary asymptotes
  to a circle of radius $a$
  in both  the distant future and  distant past?  Can we insure
  moreover that the motion is collision-free and that  throughout the motion,   $U \ge   \frac{2m}{a} $? 
  \end{oproblem} 
    We hope that one of our readers takes up this challenge and proves this or a closely related  ``Sitnikov version''
 of Moeckel's theorem \ref{Moeckel thm}.

 \vskip .3cm
 
 {\bf Acknowledgement.}  I'd like to thank Rick Moeckel
 for useful conversations.  I am also grateful to the Simon's foundation for a travel
 grant that allowed me to attend the May 2025 AMS conference in San Luis
 Obispo from which a conference proceedings was organized which led to this article.

\bibliographystyle{amsplain}

\end{document}